\documentclass[final,reqno]{amsart}
\usepackage{amsmath,amsthm,amssymb}
\usepackage{color,soul}
\usepackage{tikz}
\usepackage{mathrsfs}

\newcommand{\R}{\mathbb{R}}
\newcommand{\C}{\mathbb{C}}
\newcommand{\N}{\mathbb{N}}

\newcommand{\SH}{S_\mathcal{H}}
\newcommand{\BH}{\mathcal{B}(\mathcal{H})}
\newcommand{\DD}{\overline{\mathbb{D}}}

\newcommand{\re}{\operatorname{Re}}

\newcommand{\inner}[1]{\left\langle #1 \right\rangle}
\newcommand{\cl}[1]{\overline{#1}}

\theoremstyle{plain}
\newtheorem{theorem}{Theorem}[section]

\newtheorem{corollary}[theorem]{Corollary}
\newtheorem{lemma}[theorem]{Lemma}

\numberwithin{equation}{section}

\theoremstyle{definition}

\newtheorem{example}[theorem]{Example}

\theoremstyle{remark}
\newtheorem{remark}[theorem]{Remark}

\begin{document}
\title{Numerical ranges encircled by analytic curves}
\author[B. Lins]{Brian Lins}
\date{}
\address{Brian Lins, Hampden-Sydney College}
\email{blins@hsc.edu}


\begin{abstract}

Let $D$ be a bounded convex domain in $\C$ with a regular analytic boundary. Suppose that the numerical range $W(A)$ of a bounded linear operator $A$ is contained in $\cl{D}$. If $\cl{W(A)}$ intersects the boundary $\partial D$ at infinitely many points while the essential numerical range $W_\text{ess}(A)$ does not intersect $\partial D$, then $W(A) = \cl{D}$. This generalizes some infinite dimensional analogues of a result of Anderson \cite{BiSpSr18,GaWu06}. 

\end{abstract}

\maketitle

\section{Introduction}

Let $\mathcal{H}$ be a complex Hilbert space and let $\BH$ denote the set of bounded linear operators on $\mathcal{H}$.  Let $\SH$ denote the unit sphere in $\mathcal{H}$. For an operator $A \in \BH$, the \emph{numerical range} of $A$, denoted $W(A)$, is the range of the map $f_A: \SH \rightarrow \C$ defined by
$$f_A(x) := \inner{Ax,x}.$$
According to the Toeplitz-Hausdorff theorem, the numerical range is a convex set. It is also well known that $\cl{W(A)}$ contains the spectrum of $A$, $\sigma(A)$. In 1970 Anderson proved, but did not publish, the following striking theorem.
\begin{theorem} \label{thm:Anderson}
If the numerical range of a matrix $A \in \C^{n \times n}$ is contained in the closed unit disk $\DD$ and $W(A)$ intersects $\partial \mathbb{D}$ at more than $n$ points, then $W(A) = \DD$.  
\end{theorem}
At its heart, this theorem is an algebraic result since both the boundary of the numerical range and the unit circle are algebraic curves \cite{Kippenhahn51}. As recounted in \cite{GaWu06}, Anderson's original proof combined B\'ezout's theorem with Kippenhahn's algebraic description of the boundary curves of the numerical range.  

The following infinite dimensional analogue of Anderson's theorem was proved by Gau and Wu \cite{GaWu06}. 
\begin{theorem} \label{thm:GauWu}
If $A \in \BH$ is a compact operator with $W(A) \subseteq \DD$ and $\cl{W(A)}$ intersects $\partial \mathbb{D}$ at infinitely many points, then $W(A) = \DD$.  
\end{theorem}

This theorem lacks the algebraic character of Theorem \ref{thm:Anderson} since the boundary of $W(A)$ need not be an algebraic curve when $\mathcal{H}$ is infinite dimensional.  Instead, the proof of Theorem \ref{thm:GauWu} relies on an observation, first made by Narcowich \cite{Narcowich80}, that the boundary of the numerical range of a compact operator consists of a countable number of regular analytic curves. Recall that a curve is \emph{analytic} if it is the range of an analytic function $\gamma$ from an interval $I \subseteq \R$ into $\C$. It is \emph{regular} if $\gamma'(t) \ne 0$ for all $t \in I$. 

Recently, Theorem \ref{thm:GauWu} has been generalized to compact plus normal operators \cite{BiSpSr18}. For non-compact operators, the boundary of the numerical range need not be piecewise analytic. In fact, if $N \in \BH$ is normal, then $\cl{W(N)}$ is the convex hull of $\sigma(N)$ \cite{Berberian64}. Therefore one can choose $N$ such that $\cl{W(N)}$ is any compact convex subset of $\C$. To avoid pathological cases, Birbonshi et al. make use of the essential spectrum in \cite[Theorems 3 and 4]{BiSpSr18}. In what follows, we will use the essential numerical range for the same purpose.  By utilizing the essential numerical range rather than the essential spectrum, our results apply to more general bounded linear operators.  
Because our results are analytic rather than algebraic in nature, we will also see that any bounded convex domain with a regular analytic boundary can take the place of the unit disk.  

The \emph{essential numerical range} of $A$ is 
$$W_\text{ess}(A) := \bigcap \cl{W(A+K)},$$
where the intersection is taken over all compact operators $K \in \BH$. The essential numerical range was introduced in \cite{StWi68}. Several equivalent characterizations of $W_\text{ess}(A)$ were given in \cite{FiStWi72}, one of which is the following. 
$$W_\text{ess}(A) = \{\lambda \in \C : \exists \, x_k \in \SH \text{ with } x_k \xrightarrow{w} 0 \text{ and } \inner{Ax_k,x_k} \rightarrow \lambda \}.$$
From the definition, it is clear that $W_\text{ess}(A)$ is a nonempty, compact, convex subset of $\C$, unless $\mathcal{H}$ is finite dimensional in which case $W_\text{ess}(A) = \varnothing$. If $A$ is a compact operator on an infinite dimensional Hilbert space, then $W_\text{ess}(A) = \{0\}$. The essential numerical range $W_\text{ess}(A)$ contains the \emph{essential spectrum} $\sigma_\text{ess}(A)$, which is the set of all $\lambda \in \C$ such that $A-\lambda I$ is not Fredholm \cite{StWi68} (see also \cite{FiStWi72} for a discussion of alternative notions of essential spectra). When $A$ is self-adjoint, $\sigma(A)\backslash \sigma_\text{ess}(A) = \sigma_\text{disc}(A)$ where $\sigma_\text{disc}(A)$ is the \emph{discrete spectrum} of $A$ consisting of all isolated eigenvalues of $A$ with finite multiplicity \cite[Section VII.3]{ReedSimon}.   

Our main result is the following theorem which generalizes \cite[Theorem 3]{BiSpSr18} (see Remark \ref{rem:gen}).


\begin{theorem} \label{thm:main}
Let $A \in \BH$ and suppose that $W(A) \subseteq \cl{D}$ where $D$ is a convex set. Let $\Gamma$ be a compact regular analytic curve contained in $\partial D$ such that none of the tangent lines to $\Gamma$ intersect $W_\text{ess}(A)$ and $\Gamma$ is not a line segment. If $\cl{W(A)}$ intersects $\Gamma$ at infinitely many points, then $\Gamma \subset W(A)$.
\end{theorem}


An immediate corollary is the following infinite dimensional analogue of Anderson's theorem that generalizes both Theorem \ref{thm:GauWu} and \cite[Theorem 4]{BiSpSr18}. 

\begin{corollary} \label{cor:inf}
Let $D$ be a bounded open convex subset of $\C$ with a regular analytic boundary curve. Suppose that $W(A) \subseteq \cl{D}$ and $W_\text{ess}(A) \subset D$ for some $A \in \BH$. If $\cl{W(A)}$ intersects $\partial D$ at infinitely many points, then $W(A) = \cl{D}$. 
\end{corollary}

\section{Proof of main result}
Before proving Theorem \ref{thm:main}, we collect some known facts related to numerical ranges. 
For any operator $B \in \BH$, recall that the \emph{real part of} $B$ is the self-adjoint operator $\re(B) = \frac{1}{2}(B+B^*)$. The following result is well known. See, for example, \cite[Theorem 3.1]{Narcowich80} and its proof.  

\begin{lemma} \label{lem:support}
Let $A \in \BH$ and let $\mu(\theta)$ denote the maximum of the spectrum of $\re(e^{-i \theta}A)$ for all $\theta \in \R$.  Then,
$$\cl{W(A)} = \bigcap_{0 \le \theta < 2\pi} \{z \in \C \, : \, \re(e^{-i \theta}z) \le \mu(\theta) \}.$$
Furthermore, any $z \in \partial W(A)$ lies on a support line $L_\theta$ where
$$L_\theta := \{ z \in \C \, : \, \re(e^{-i\theta} z) = \mu(\theta) \}$$
for some $\theta \in [0,2\pi)$.  
\end{lemma}

The following lemma uses the same notation as Lemma \ref{lem:support}. We don't claim that this is a new result, but it doesn't appear to be widely known, so we include a proof (see also \cite[Lemma 4.1]{LiSp18}).

\begin{lemma} \label{lem:ess}
There is a point $z \in W_\text{ess}(A) \cap L_\theta$ if and only if $\mu(\theta)$ is in the essential spectrum of $\re(e^{-i \theta}A)$.  
\end{lemma}

\begin{proof}
Since $W(e^{-i\theta}A) = e^{-i \theta} W(A)$ and $W_\text{ess}(e^{-i \theta} A) = e^{-i \theta} W_\text{ess}(A)$, we may assume without loss of generality that $\theta = 0$. Suppose that $z \in W_\text{ess}(A) \cap L_0$.  There is a sequence $x_k \in \SH$ such that $x_k$ converges weakly to zero, and $\inner{Ax_k,x_k} \rightarrow z$.  By Lemma \ref{lem:support}, $\re \inner{Ax_k, x_k} = \inner{\re(A)x_k,x_k}$ converges to $\mu(0)$. Thus $\mu(0) \in W_\text{ess}(\re(A))$.  For a self-adjoint operator, the essential numerical range is the convex hull of the essential spectrum \cite[Corollary 1]{StWi68}, therefore $\mu(0)$ is in the essential spectrum of $\re(A)$.   

Conversely, if $\mu(0) \in \sigma_\text{ess}(\re(A))$, then $\mu(0) \in W_\text{ess}(\re(A))$ so there is a sequence $x_k \in \SH$ such that $x_k \xrightarrow{w} 0$ and $\inner{\re(A)x_k, x_k} \rightarrow \mu(0)$. By passing to a subsequence, we may assume that $\inner{Ax_k,x_k}$ converges to some $z \in L_0$ and since the sequence $x_k$ converges weakly to zero, $z \in W_\text{ess}(A)$. 
\end{proof}

The final crucial ingredient of our proof is the perturbation theory for eigenvalues of an analytic family of self-adjoint operators. First introduced by Rellich \cite{Rellich37} and later refined by Kato \cite{Kato} and Sz.-Nagy \cite{SzNagy46}, this theory was used by Narcowich \cite{Narcowich80} to show that the boundary of the numerical range of a compact operator is a countable union of regular analytic arcs. In fact, although it is not explicitly stated, the results in \cite{Narcowich80} imply that for $A \in \BH$, any closed arc of $\partial W(A)$ whose support lines do not intersect $W_\text{ess}(A)$ is a finite union of regular analytic curves that are contained in $W(A)$ \cite[Proposition 4.2]{LiSp18}. For a readable introduction to analytic perturbation theory and a proof of the following theorem, see \cite[Section 136]{RieszNagy}.
\begin{theorem} \label{thm:kato}
Suppose that $A(t)$ is a self-adjoint operator valued function of a real parameter $t$ given by
$$A(t) = A^{(0)} + A^{(1)}t + A^{(2)}t^2 + \ldots$$ 
where the coefficients $A^{(k)} \in \BH$ are self-adjoint and $\sum_{k\in \N} t^k \|A^{(k)}\|$ converges for all $t$ in a neighborhood of $0$. If $A(0)$ has an isolated eigenvalue $\lambda^{(0)}$ of finite multiplicity $m$, then there exists $\epsilon >0$ such that when $-\epsilon < t < \epsilon$, the spectrum of $A(t)$ in a neighborhood of $\lambda^{(0)}$ consists of $m$ (not necessarily distinct) real values $\lambda_j(t)$ given by 
$$\lambda_j(t) = \lambda^{(0)} + \lambda_j^{(1)}t + \lambda_j^{(2)}t^2 + \ldots.$$
Furthermore there are $m$ corresponding orthogonal $\varphi_j(t) \in \SH$ given by
$$\varphi_j(t) = \varphi_j^{(0)} + \varphi_j^{(1)}t + \varphi_j^{(2)}t^2 + \ldots$$
such that $A(t) \varphi_j(t) = \lambda_j(t) \varphi_j(t)$ for all $j$ and $-\epsilon < t < \epsilon$.  
\end{theorem}

\begin{proof}[Proof of Theorem \ref{thm:main}]
If $\cl{W(A)}$ intersects $\Gamma$ at infinitely many points, then there is a sequence of distinct points $z_k \in \cl{W(A)} \cap \Gamma$ that converge to some $z$.  Since both $\cl{W(A)}$ and $\Gamma$ are closed sets, $z \in \cl{W(A)} \cap \Gamma$. The tangent line to $\Gamma$ at $z$ is a support line for the convex set $\cl{D}$, and since $W(A) \subseteq \cl{D}$, it must also be a support line for $\cl{W(A)}$.  We may assume without loss of generality that this line is $L_0$, as defined in Lemma \ref{lem:support}. By assumption, $L_0$ does not intersect $W_\text{ess}(A)$. Therefore the maximum element of the spectrum of $\re(A)$, which is $\re(z)$, is an isolated eigenvalue with finite multiplicity by Lemma \ref{lem:ess}. 

We can now apply Theorem \ref{thm:kato} to the analytic family of self-adjoint operators $\re(e^{-i\theta}A)$. There is a finite collection of real analytic functions $\lambda_j(\theta)$ defined in an $\epsilon$-neighborhood of $\theta = 0$ such that $\lambda_j(0) = \re(z)$ for all $j$, and each $\lambda_j(\theta)$ is an eigenvalue of $\re(e^{-i\theta}A)$ for all $-\epsilon < \theta < \epsilon$.  All other elements of the spectrum of $\re(e^{-i\theta}A)$ are strictly less than the minimum $\lambda_j(\theta)$ when $\theta \in (-\epsilon,\epsilon)$. To each $\lambda_j(\theta)$, there is a corresponding eigenvector $\varphi_j(\theta) \in \SH$, and these eigenvectors are also analytic functions of $\theta$ in the interval $(-\epsilon,\epsilon)$. Each pair $(\lambda_j, \varphi_j)$ has an associated analytic curve $\zeta_j(\theta) := f_A(\varphi_j(\theta))$ contained in $W(A)$. Each of the curves $\zeta_j(\theta)$ is either regular or its range is a single point \cite[Lemma 3.1]{Narcowich80}. 

The boundary of $W(A)$ in the vicinity of $z$ consists of either a single analytic curve corresponding to one of the $\zeta_j(\theta)$, or to two analytic curves which can either be line segments with one endpoint at $z$ or curves that correspond to some of the $\zeta_j(\theta)$ \cite[Theorem 5.1]{Narcowich80}.  Since the sequence $z_k$ approaches $z$ along the boundary of $W(A)$, and the $z_k$ are not contained in a line segment, we conclude that there are infinitely many $z_k$ contained in an arc of $\partial W(A)$ that is parametrized by one of the functions $\zeta_j(\theta)$. This means that the curve given by $\zeta_j(\theta)$ intersects $\Gamma$ at infinitely many points.  

Both $\zeta_j(\theta)$ and $\Gamma$ are regular analytic curves, and both have a vertical tangent line at $z$, so we may parameterize both curves analytically in a neighborhood of $z$ using their imaginary coordinate as the common parameter. Then, since the two analytic curves intersect infinitely many times and the intersection points accumulate at $z$, we conclude that the two curves are identical in a neighborhood of $z$ (possibly one-sided if $z$ is an endpoint of $\Gamma$).  Every point of $\Gamma$ in that neighborhood also lies on the curve $\zeta_j(\theta)$ for some $\theta$, and is therefore contained in $W(A)$. 
We can repeat this argument at the endpoint(s) of the neighborhood where $\zeta_j(\theta)$ and $\Gamma$ coincide to analytically continue the functions $\lambda_j$, $\varphi_j$, and $\zeta_j$ until the range of $\zeta_j(\theta)$ contains all of $\Gamma$. Therefore $\Gamma \subset W(A)$.  
\end{proof}

\begin{remark} \label{rem:gen}
If $A = N + K$ where $N \in \BH$ is normal and $K$ is compact, then $W_\text{ess}(A) = W_\text{ess}(N)$.  The essential numerical range of a normal operator is the convex hull of its essential spectrum, and this essential spectrum does not change when $N$ is replaced by $N + K$ \cite{FiStWi72}. Therefore the essential numerical range of $A$ is the convex hull of the essential spectrum of $A$.  So Theorem \ref{thm:main} implies \cite[Theorems 3 and 4]{BiSpSr18}. 
\end{remark}

\section{Examples}

The following examples illustrate the how the conditions of Theorem \ref{thm:main} and Corollary \ref{cor:inf} can fail.
\begin{example}
It is easy to construct a normal operator $A \in \BH$ such that $W(A)$ is contained in the closed unit circle $\cl{\mathbb{D}}$ and $W(A)$ intersects $\partial \mathbb{D}$ at infinitely many points without $W(A)$ equaling $\cl{\mathbb{D}}$.  For example, take a diagonal operator on $\ell_2(\N)$ with diagonal entries $e^{i/k}$, $k \in \N$.  For such an operator, the accumulation point of $\cl{W(A)} \cap \partial \mathbb{D}$ is contained in the essential numerical range $W_\text{ess}(A)$, and therefore the conditions of of Theorem \ref{thm:main} and Corollary \ref{cor:inf} are not satisfied. 
\end{example}

\begin{example}
A weighted shift operator $S: \ell_2(\N) \rightarrow \ell_2(\N)$ is defined by $Se_k := s_k e_{k+1}$ where $\{e_k\}$ is the elementary basis of $\ell_2(\N)$ and $s_k$ is a bounded sequence of weights.  It is possible to choose the $s_k$ so that the numerical range $W(S)$ is the open unit disk $\mathbb{D}$ \cite[Proposition 6]{Stout83}.  In that case, $\cl{W(S)}$ intersects every point on $\partial \mathbb{D}$ but Corollary \ref{cor:inf} does not apply since every point in $\partial \mathbb{D}$ is an element of the essential numerical range $W_\text{ess}(S)$ by \cite[Theorem 1]{Lancaster75}.  
\end{example}

The requirement that $\Gamma$ not be a line segment in Theorem \ref{thm:main} is essential, as the following two examples show. 
\begin{example}
Suppose that $D$ is the convex hull of $\{0, 2, 2i\}$ and $A$ is the 3-by-3 diagonal matrix with diagonal entries 0, 1, and $i$. Clearly, $W(A) \subset D$. The line segment $[0,2]$ is a regular analytic curve contained in $\partial D$ that intersects $W(A)$ at infinitely many points but is not contained in $W(A)$.  
\end{example}

\begin{example}
Consider the compact diagonal operator $A$ on $\ell_2(\N)$ with diagonal entries, $1$ and $i/k$ for $k \in \N$.  Then $\cl{W(A)}$ is the convex hull of $\{0, 1, i\}$, but the line segment $[0,1)$ is not in $W(A)$. 
\end{example}

\color{black}

\bibliography{perturb}
\bibliographystyle{plain}

\end{document}